\documentclass{amsart}

\usepackage[capitalise]{cleveref}
\usepackage{stmaryrd}

\newcommand{\qf}[1]{{\langle{#1}\rangle}}
\newcommand{\bpf}[1]{{\langle\!\langle}{#1}\rangle\!\rangle}
\newcommand{\pf}[1]{{\langle\!\langle{#1}]\!]}} 
\newcommand{\C}{\mathcal{C}}
\newcommand{\mc}{{\mathbb{C}}}
\newcommand{\mq}{{\mathbb{Q}}}
\newcommand{\mh}{{\mathbb {H}}}
\newcommand{\mz}{{\mathbb {Z}}}

\DeclareMathOperator{\Nrd}{Nrd}
\DeclareMathOperator{\dlog}{dlog}

\newtheorem{lem}{Lemma}[section]
\newtheorem{prop}[lem]{Proposition}
\newtheorem{thm}[lem]{Theorem}
\newtheorem{cor}[lem]{Corollary}

\theoremstyle{remark}
\newtheorem{rem}[lem]{Remark}

\newtheorem{example}[lem]{Example}

\title[]{Minimal quadratic forms for the function field of a conic in characteristic $2$}   

\author{Adam Chapman} 
\address{School of Computer Science, Academic College of Tel-Aviv-Yaffo, Rabenu Yeruham St., P.O.B 8401 Yaffo, 6818211, Israel}
\email{adam1chapman@yahoo.com}
\thanks{The first author acknowledges the receipt of the Chateaubriand Fellowship (969845L) offered by the French Embassy in Israel, and thanks LAGA, Universit\'e Sorbonne Paris Nord, for the hospitality in Fall 2020. }

\author{Anne Qu\'eguiner-Mathieu}
\address{Universit\'e Sorbonne Paris Nord, Institut Galilée, 
LAGA - UMR 7539 du CNRS,  
93430 Villetaneuse, France}
\email{queguin@math.univ-paris13.fr}
\thanks{Both authors thank Jean-Pierre Tignol for useful conversations on $v$-value functions and norms.} 


\begin{document}

\begin{abstract}
In this note, we construct explicit examples of $F_Q$-minimal quadratic forms of dimension $5$ and $7$, where $F_Q$ is the function field of a conic over a field $F$ of characteristic $2$. The construction uses the fact that any set of $n$ cyclic $p$ algebras over a field of characteristic $p$ can be described using only $n+1$ elements of the base field. It also uses a general result that provides an upper bound on the Witt index of an orthogonal sum of two regular anisotropic quadratic forms over a henselian valued field. 
\end{abstract} 

\maketitle

\section{Introduction} 

Let $Q$ be a quaternion division algebra over a field $F$. Its Severi-Brauer variety is a conic, which is a codimension $1$ subform of the norm form $n_Q$ of $Q$. We denote by $F_Q$ the function field of this conic. 
One may try to describe anisotropic quadratic forms that become either hyperbolic or isotropic over this function field $F_Q$. 
For hyperbolicity, there is a well-known answer : since $n_Q$ is a $2$-fold Pfister form, the subform theorem~\cite[Cor. 23.6]{EKM} asserts that an anisotropic quadratic form becomes hyperbolic over $F_Q$ if and only if it is a multiple of $n_Q$. The question is much harder for isotropy. 
Clearly, any quadratic form having the conic as a subform becomes isotropic over $F_Q$. But the converse does not hold. 
This was initially noticed by Wadsworth, and led to the notion of minimal form, introduced by Hoffmann: a quadratic form $\varphi$ over $F$ is called $F_Q$-minimal if $\varphi$ is anisotropic, 
$\varphi_{F_Q}$ is isotropic, and 
$\varphi$ does not contain any strict subform which becomes isotropic over $F_Q$.
Minimal forms in characteristic different from $2$ were studied in the 90's by Hoffmann, Lewis and Van Geel, see
~\cite{Hoffmann:thesis},~\cite{LewisVanGeel:94},~\cite{HoffmannLewisVanGeel:1992} and~\cite{HoffmannVanGeel:1995}. 
In particular, they described a field and a conic for which there exist minimal forms of any odd dimension $\geq 3$, see~\cite[Thm. 4.6]{HoffmannVanGeel:1995}. 

Large parts of this theory extend to the characteristic $2$ case, as was shown by Faivre in~\cite{Faivre:thesis}. In particular, he defined an invariant, called the minimal splitting index $\bar{s}_{{F_Q}}(\varphi)$, which can be used to characterize $F_Q$-minimal forms. For forms of dimension $5$, he gave another characterization using the Clifford algebra. Nevertheless, his thesis does not contain any explicit example of $F_Q$-minimal forms. In this note, we construct such examples in dimension $5$ and $7$, see~\cref{dim5.thm} and~\cref{dim7.thm}, which are the first to appear in the literature in characteristic 2. 

Recall that the even Clifford algebra of a quadratic space $(V,q)$ is naturally endowed with an involution, called the canonical involution, which is of symplectic type if the base field has characteristic $2$ and $V$ is odd dimensional, see~\cite[\S 8.A]{BOI}. Moreover, a central simple algebra with symplectic involution $(A,\sigma)$ is called hyperbolic if $A$ contains an isotropic ideal of maximal dimension, that is a right ideal $I\subset A$ such that $\sigma(I)I=\{0\}$ and $\dim_F(I)=\frac 12\dim_F(A)$~\cite[(6.8)]{BOI}. Starting from an $F_Q$-minimal quadratic form of dimension $5$, the even Clifford algebra functor produces an example of a degree $4$ algebra with symplectic involution which is non hyperbolic, becomes hyperbolic over $F_Q$, and does not contain any subalgebra with involution  isomorphic to $(Q,\bar{\ \,})$, where $\bar{\ }$ denotes the canonical involution. This was already observed in~\cite{QueguinerTignol:2010} in characteristic different from $2$ and in~\cite{DolphinQueguiner:2017} in characteristic $2$. So our paper also provides explicit examples of such algebras with involution in characteristic $2$. 

It is true in general that the even Clifford algebra of a quadratic form decomposes as a tensor product of quaternion algebras with involution. Over a field of characteristic $2$, any algebra which decomposes as a tensor product of $n$ quaternion algebras may be described in a specific way, using only $n+1$ elements of the base field. This fact, which actually holds over any field of positive characteristic for tensor product of cyclic $p$ algebras, see~\cref{ncyclic}, has been a crucial ingredient in designing our examples of minimal forms. 

The base fields in those examples are iterated Laurent series field. Over such a field, any vector space endowed with a regular quadratic form admits a $v$-value function which is a norm. To prove the $F_Q$-minimality of the $7$-dimensional example, we use a general result of independant interest and valid in arbitrary characteristic, which provides an upper bound for the Witt index of an orthogonal sum of two regular quadratic forms over a henselian valued field, see~\cref{WittIndex.st}. 

The content of the paper is as follows. In~\cref{cyclic.sec}, we provide a general description of any set of $n$ cyclic $p$ algebras over a field of characteristic $p$, where $p$ is an arbitrary prime number. The bound on the Witt index of an orthogonal sum of two regular quadratic forms over a henselian valued field of arbitrary characteristic is in~\cref{quadform.sec}. Finally, the examples of $F_Q$-minimal forms over a base field of characteristic $2$ are described in~\cref{minimal.sec}. 

\subsection*{Notations}
Throughout the paper, $F$ is a field of positive characteristic $p$, except in~\cref{quadform.sec} where the characteristic is arbitrary; in Section \ref{minimal.sec} on minimal quadratic forms, we will assume $p=2$. 

For $\alpha\in F$ and $\beta\in F^\times$, we let $[\alpha,\beta)_{p,F}$ be the cyclic $p$-algebra generated over $F$ by two elements $i$ and $j$ subject to the relations $i^p-i=\alpha$, $j^p=\beta$ and $jij^{-1}=i-1$. 
It follows from the definition that $[\alpha+x^p-x,\beta y^p)_{p,F}\simeq [\alpha,\beta)_{p,F}$ for all $x\in F$ and $y\in F^\times$. 
Recall that $[\alpha,\beta)_{p,F}$ is either division or isomorphic to $M_p(F)$. In particular, for all $\alpha\in F$ and $\beta\in F^\times$, we have 
$[\alpha,1)_{p,F}\simeq [0,\beta)_{p,F}\simeq [\beta,\beta)_{p,F}\simeq M_p(F).$ In the sequel, we will use the following, which is proved in~\cite[Lemma 2.3(a)]{Chapman:2017}: 
\begin{equation}
\label{easyid}
[\alpha,\beta)_{p,F}\otimes_F [\gamma,\delta)_{p,F}\simeq [\alpha+\gamma,\beta)_{p,F}\otimes [\gamma,\beta^{-1}\delta)_{p,F},
\end{equation}
for all $\alpha,\gamma\in F$ and $\beta,\delta\in F^\times$. 

We denote by $\Omega_F^*$ the ring of differential forms. It contains an element $d\beta$ for all $\beta\in F$, and we let $\dlog(\beta)=\beta^{-1}\, d\beta$ for all $\beta\not =0$. 
The subspace spanned by $n$-fold differentials $\Omega_F^n$ admits a well-defined Artin-Schreier map $\wp : \Omega_F^n  \rightarrow \Omega_F^n / d \Omega_F^{n-1}$ given by
$$ \alpha \dlog(\beta_1) \wedge \dots \wedge \dlog(\beta_n) \mapsto (\alpha^p-\alpha) \dlog(\beta_1) \wedge \dots \wedge \dlog(\beta_n).$$
The cokernel of this map is the Kato-Milne cohomology group $H_p^{n+1}(F)$.
For $n=0$ one recovers the Artin-Schreier group $H_p^1(F)=F/\wp(F)$ which describes cyclic degree $p$ extensions of $F$. For $n=1$, there is an isomorphism to the $p$-torsion of the Brauer group $H_p^2(F)\simeq {_p}{\mathrm{Br}}(F)$ (see \cite[Section 9]{GS}) given by $\alpha \dlog(\beta) \mapsto [\alpha,\beta)_{p,F}$, where we use the same notation for the Brauer class of the cyclic algebra $[\alpha,\beta)_{p,F}$ as for the algebra itself. So, when viewed as an element of the Brauer group, the symbol $[\alpha,\beta)_{p,F}$ is additive in $\alpha$ and multiplicative in $\beta$, that is 
\[[\alpha_1+\alpha_2,\beta)_{p,F}= [\alpha_1,\beta)_{p,F}+ [\alpha_2,\beta)_{p,F},\mbox{ and } [\alpha,\beta_1\beta_2)_{p,F}= [\alpha,\beta_1)_{p,F}+[\alpha,\beta_2)_{p,F},\] 
for all $\alpha, \alpha_1,\alpha_2\in F$ and $\beta, \beta_1, \beta_2\in F^\times$. 
Using these relations, one may also check, the following identity: 
\begin{equation}
\label{mainid}
[\alpha,\beta)_{p,F}\simeq [(\beta+x^p)\alpha\beta^{-1},\beta+x^p)_{p,F},
\end{equation} 
for all $\alpha\in F$, $\beta\in F^\times$ and $x\in F$ such that $\beta+x^p\not=0$. 
Indeed, by definition, we have $(\beta+x^p)\dlog(\beta+x^p)=d(\beta+x^p)$. It follows that  $(\beta+x^p)\alpha\beta^{-1}\dlog(\beta+x^p)=\alpha\beta^{-1}d(\beta+x^p)=\alpha\beta^{-1}d\beta=\alpha\dlog(\beta)$, as required. See also~\cite[Lemma 2.2(d)]{Chapman:2017} for a direct argument. 

Recall that when $p=2$, there is an isomorphism $H_2^{n+1}(F)=I_q^{n+1} F/I_q^{n+2} F$ where $I_q(F)$ is the fundamental ideal of the Witt group of nonsingular quadratic forms over $F$, as proved by Kato in~\cite{Kato:1982}. The symbol 
$\alpha \dlog(\beta_1) \wedge \dots \wedge \dlog(\beta_n)$ corresponds under this isomorphism to the $n$-fold Pfister form $\langle \! \langle \beta_n,\dots,\beta_1,\alpha]\!]$, which is, by definition, the tensor product of the $(n-1)$-fold bilinear Pfister form $\langle \! \langle \beta_n,\dots,\beta_1\rangle \! \rangle$ with the rank $2$ non singular quadratic form $[1,\alpha]$. If $n=2$, we get the quadratic form $\pf{\beta,\alpha}=[1,\alpha]\perp \beta[1,\alpha]$ which is the norm form $n_Q$ of the quaternion algebra $Q=[\alpha,\beta)_{2,F}$. 

Let $(A,\sigma)$ be a central simple algebra, endowed with a symplectic involution. A right ideal $I\subset A$ is called isotropic (with respect to the involution $\sigma$) if $\sigma(x)y=0$ for all $x,y\in I$. An isotropic ideal has dimension at most $\frac 12 \dim_F(A)$. The involution $\sigma$ is called isotropic (respectively hyperbolic) if there exists a non trivial isotropic ideal $I\subset A$ (respectively an isotropic ideal of dimension $\frac 12 \dim_F(A)$). Over a field of characteristic $2$, this definitions apply in particular to the Clifford algebra and the even Clifford algebra of an odd dimensional quadratic space, since its canonical involution is symplectic, see~\cite[6.A, 6.B, 8.A]{BOI}. 

\section{On sets and tensor products of cyclic $p$-algebras over a field of characteristic $p$}
\label{cyclic.sec}


The main result in this section is the following :
\begin{thm}
\label{ncyclic}
Let $F$ be a field of positive characteristic $p$. Any $n$ cyclic $p$-algebras can be written as 
\[[a_0,a_1)_{p,F},[a_1,a_2)_{p,F},\dots, [a_{n-1},a_n)_{p,F},\]
for some non-zero elements $a_i\in F^\times$. 
\end{thm} 
The proof uses the following lemma: 
\begin{lem}
\label{2cyclic}
Let $A_1=[\alpha,\beta)_{p,F}$ and $A_2=[\gamma,\delta)_{p,F}$ be two cyclic division $p$-algebras. 
Then, there exist $a,b\in F^\times$ such that $A_1=[a,b)_{p,F}$ and $A_2=[b,\delta)_{p,F}$. 
\end{lem} 
\begin{proof}[Proof of~\cref{2cyclic}]
Let $x=\gamma-\beta$, so that $\gamma+x^p-x=\beta+x^p$ and denote by $b$ this element. 
Since $A_1$ is division, $\beta\not\in F^{\times p}$, so that $b=\beta+x^p\not =0$; hence, by~\eqref{mainid}, the algebra $A_1$ is given by 
\[A_1=[(\beta+x^p)\alpha\beta^{-1},\beta+x^p)_{p,F}=[a,b)_{p,F},\]
where $a=(\beta+x^p)\alpha\beta^{-1}$. Again since $A_1$ is division we have $a\in F^\times$; 
moreover, we have 
\[A_2=[\gamma+x^p-x,\delta)_{p,F}=[b,\delta)_{p,F}\] 
and this finishes the proof. 
\end{proof}
\begin{proof}[Proof of~\cref{ncyclic}]
We first prove the result when all algebras are division, by induction on their number $n$. It clearly holds if $n=1$. 
Given $n$ cyclic $p$-algebras $A_1,\dots, A_n$, the induction hypothesis provides $a_1',a_2,\dots, a_n\in F^\times$ such that $A_2=[a_1',a_2)_{p,F}$ and $A_i=[a_{i-1},a_i)_{p,F}$ for $i\in\llbracket3;n\rrbracket$. 
~\cref{2cyclic} applied to $A_1$ and $A_2$ now provides $a_0$ and $a_1\in F^\times$ such that $A_1=[a_0,a_1)_{p,F}$ and $A_2=[a_1,a_2)_{p,F}$, and this concludes the proof in this case. 

\noindent If the algebras are not all division, one may first write the division algebras of the considered set as $[a_0,a_1)_{p,F}, \dots, [a_{k-1},a_k)_{p,F}$. Since $[a_k,1)_{p,F}=[1,1)_{p,F}=M_p(F)$, the result holds with $a_i=1$ for $i\in\llbracket k+1; n\rrbracket$. 
\end{proof} 

As an immediate corollary, we get 
\begin{cor}
Let $A$ be a tensor product of $n$ cyclic $p$-algebras. 
There exists $a_0,\dots, a_n\in F^\times$ such that \[A=[a_0,a_1)_{p,F}\otimes[a_1,a_2)_{p,F}\otimes \dots\otimes[a_{n-1},a_n)_{p,F}.\]
\end{cor}

\begin{rem}
\cref{ncyclic} provides an easy proof of the fact that the essential dimension  of decomposable algebras of degree $p^n$ and exponent $p$ is at most $n+1$. This follows from the essential dimension of the group $(\mathbb{Z}/p\mathbb{Z})^n$ over a field of characteristic $p$, which is 1 (see \cite{Ledet:2004}), but the argument there is more elaborate. See \cite{McKinnie:2017} for a detailed discussion on the essential dimension of such algebras.
\end{rem}

\section{Witt index of an orthogonal sum of two anisotropic quadratic forms over a henselian valued field} 
\label{quadform.sec}
The main result in this section provides an upper bound for the Witt index of an orthogonal sum of two anisotropic quadratic forms over a henselian valued field, under some condition on the value function associated to one of them. We refer the reader to~\cite[\S 3.1.1]{TignolWadsworth:2015} for general information on value functions and norms. 

Let $F$ be a field of arbitrary characteristic, and $V$ a finite-dimensional vector space over $F$. 
We assume $F$ is endowed with a valuation $v:\,F\rightarrow \Gamma\cup\{\infty\}$, where $\Gamma$ is a totally ordered group. Replacing $\Gamma$ by its divisible hull if necessary, we may assume $\Gamma$ is divisible, 
and we denote the value group of $F$ by $\Gamma_F=v(F^\times)\subset \Gamma$. 
Recall that a $v$-value function on $V$ is a map $\alpha:\,V\rightarrow \Gamma\cup\{\infty\}$ such that 
\begin{enumerate}
\item $\alpha(x)=0$ if and only if $x=0$; 
\item $\alpha(x+y)\geq \mbox{min}\bigl(\alpha(x),\alpha(y)\bigr)$ for all $x,y\in V$;
\item $\alpha(x\lambda)=\alpha(x)+v(\lambda)$ for all $x\in V$ and $\lambda\in F$. 
\end{enumerate}
The value set of $\alpha$, denoted by $\Gamma_\alpha=\{\alpha(x),\ x\in V,\ x\not =0\}\subset \Gamma$ need not be a subgroup of $\Gamma$. Nevertheless, condition (3) above shows that $\Gamma_\alpha$ consists of a union of cosets of $\Gamma_F$, and we denote by $|\Gamma_\alpha:\Gamma_F|$ the number of cosets.  
The $v$-value function $\alpha$ is called a $v$-norm if $V$ admits a splitting base for $\alpha$, that is a base $(e_i)_{1\leq i\leq n}$ such that for all $\lambda_1,\dots,\lambda_n\in F$, we have 
\[\alpha(\sum_{i=1}^n e_i\lambda_i)=\mbox{min}\{\alpha(e_i)+v(\lambda_i),\ 1\leq i\leq n\}.\]
The value set $\Gamma_\alpha$ then consists of the cosets $\alpha(e_i)+\Gamma_F$ for $1\leq i\leq n$. 
Hence, we have $|\Gamma_\alpha:\Gamma_F|\leq \dim_F(V)$, with equality if and only if these cosets are pairwise distinct. 

Let us assume now that $V$ is endowed with an anisotropic quadratic form $\varphi$. 
As explained in~\cite[\S 4]{ET:Springer}, see also~\cite{Springer:1956}, the function $\alpha:\, V\rightarrow \Gamma\cup\{\infty\}$ defined by $\alpha(x)=\frac 12 v\bigl(\varphi(x)\bigr)$ for all $x\in V$ is a $v$-value function on $V$. We call it the $v$-value function associated to $\varphi$. 
Our construction of examples of $7$-dimensional $F_Q$-minimal forms uses the following : 
\begin{prop}
\label{WittIndex.st}
Let $(V,\varphi)$ and $(W,\psi)$ be two anisotropic quadratic spaces over a henselian valued field $(F,v)$, such that the orthogonal sum $\varphi\perp\psi$ is regular.  Denote by $\alpha$ and $\beta$ the $v$-value functions on $V$ and $W$ respectively associated to $\varphi$ and $\psi$. If $\alpha$ is a norm and $|\Gamma_\alpha:\Gamma_F|=\dim_F(V)$, then the Witt index of the orthogonal sum $\varphi\perp\psi$ satisfies 
\[i_w(\varphi\perp\psi)\leq |(\Gamma_\alpha\cap\Gamma_\beta):\Gamma_F|.\]
\end{prop} 
\begin{rem}
Since $\varphi$ and $\psi$ are anisotropic, both are regular. 
In characteristic not $2$, it follows their orthogonal sum also is regular, but this is not true in general in characteristic $2$. 
We need this additional assumption in the statement. 
\end{rem}
\begin{proof}
Let $Z$ be a maximal totally isotropic subspace of $(V\oplus W,\varphi\perp\psi)$. 
Since $\varphi\perp\psi$ is regular, the dimension of $Z$ is equal to the Witt index of this form, and we need to prove that $\dim_F(Z)\leq |(\Gamma_\alpha\cap\Gamma_\beta):\Gamma_F|.$ Denote by $p_1$ and $p_2$ the projections of $V\oplus W$ on $V$ and $W$ respectively; since $Z$ is totally isotropic, for all $z\in Z$, we have 
\begin{equation}
\label{isot}
\bigl(\varphi\perp\psi\bigr)(z)=\varphi\bigl(p_1(z)\bigr)+\psi\bigl(p_2(z)\bigr)=0.
\end{equation}
On the other hand, $\varphi$ and $\psi$ are both anisotropic. Hence, the projection $p_1$ induces an isomorphism between $Z$ and $Z_1=p_1(Z)\subset V$. Similarly, $p_2$ induces an isomorphism between $Z$ and $Z_2=p_2(Z)\subset W$. Combining those isomorphisms and using~\eqref{isot}, we get an isometry $f$ between the quadratic spaces $(Z_1,\varphi_{|Z_1})$ and $(Z_2,-\psi_{|Z_2})$. By definition of the $v$-value functions $\alpha$ and $\beta$ associated to $\varphi$ and $\psi$ respectively, it follows that for all $z_1\in Z_1$ we have 
\[\alpha(z_1)=\frac 12 v\bigl(\varphi(z_1)\bigr)=\frac12 v\Bigl(-\psi\bigl(f(z_1)\bigr)\Bigr)=\frac12 v\Bigl(\psi\bigl(f(z_1)\bigr)\Bigr)=\beta\bigl(f(z_1)\bigr).\]
Therefore, since $f$ is bijective, the $v$-value functions $\alpha$ and $\beta$ induce $v$-value functions on $Z_1$ and $Z_2$ with the same value group, which we denote by $\Gamma_Z$. It satisfies  \[\Gamma_Z=\{\alpha(z_1),\ z_1\in Z_1,\ z_1\not=0\}=\{\beta(z_2),\ z_2\in Z_2,\ z_2\not=0\}\subset \Gamma_\alpha\cap \Gamma_\beta,\] 
and it follows that \[|\Gamma_Z:\Gamma_F|\leq  |(\Gamma_\alpha\cap\Gamma_\beta):\Gamma_F|.\]

By~\cite[Prop. 3.14]{TignolWadsworth:2015}, since $\alpha$ is a norm, its restriction to $Z_1$ also is a norm, hence in particular we have 
$|\Gamma_Z:\Gamma_F|\leq\dim_F(Z_1)$. We claim they actually are equal. Indeed, assume for the sake of contradiction that  $|\Gamma_Z:\Gamma_F|<\dim_F(Z_1)$. Again by~\cite[Prop. 3.14]{TignolWadsworth:2015}, $Z_1$ admits a splitting complement $Y$, that is $V=Z_1\oplus Y$ and $\alpha(v)=\min\bigr(\alpha(z_1),\alpha(y)\bigl)$ for all $v=z_1+y\in V$ with $z_1\in Z_1$ and $y\in Y$. 
Therefore, we have $\Gamma_\alpha=\Gamma_{\alpha|Z_1}\cup \Gamma_{\alpha|Y}=\Gamma_Z\cup \Gamma_{\alpha|Y}$. Moreover, the restriction of $\alpha$ to $Y$ also is a norm, so that its value set satisfies $|\Gamma_{\alpha|Y}:\Gamma_F|\leq \dim_F(Y)$. 
So, we get 
\[|\Gamma_\alpha:\Gamma_F|\leq |\Gamma_{Z}:\Gamma_F|+|\Gamma_{\alpha|Y}:\Gamma_F|<\dim_F(Z_1)+\dim_F(Y)=\dim_F(V),\]
which contradicts the assumption we made on $\alpha$ in the statement. 
So we have $|\Gamma_Z:\Gamma_F|=\dim_F(Z_1)=\dim_F(Z)$, and this concludes the proof. 
\end{proof}

\section{Explicit examples of $F_Q$-minimal forms} 
\label{minimal.sec}
From now on, we assume $p=2$, so that $F$ has characteristic $2$, and we use the notation $[\alpha,\beta)=[\alpha,\beta)_{2,F}$ for all $\alpha\in F$ and $\beta\in F^\times$. 
The purpose of this section is to construct explicit examples 
of $F_Q$-minimal forms of dimension $5$ and $7$. In both cases, the results are proved using characterisations of $F_Q$-minimal forms in characteristic $2$ which are due to Faivre~\cite[\S 5.2]{Faivre:thesis}. As explained below, the method is slightly different in dimension $5$ and in dimension $7$.

We first prove the following: 
\begin{prop} 
\label{dim5.thm}
Assume $a,b,c$ are elements of $F$ such that $a\dlog(b)\wedge\dlog(c)$ is a non trivial symbol in $H^3_2(F)$, and let $Q$ be the quaternion algebra $Q=[a,c)$. 
The quadratic form $\varphi_1=c[1,a+b]\perp b[1,a]\perp \qf{1}$ is anisotropic, and becomes isotropic over $F_Q$. If in addition $[a,b)\otimes_F [b,c)$ is division over $F$, then $\varphi_1$ is $F_Q$-minimal. 
\end{prop} 

\begin{proof}
Let $F$, $a,b,c$ and $\varphi_1$ be as in the statement. 
Since the quaternion algebras $[a+b,b)$ and $[a,b)$ are isomorphic, they have isometric norm form $\pf{b,a+b}\simeq \pf{b,a}$. It follows that 
\[\varphi_1\perp bc[1,a+b]\simeq c\pf{b,a}\perp b[1,a]\perp\qf{1}.\] 
Therefore, one may easily check that $\varphi_1$ is isometric to a subform of $\pf{b,c,a}$. 
This form has non trivial Arason invariant $a\dlog(b)\wedge\dlog(c)$, hence it is anisotropic and so is $\varphi_1$. Moreover $\pf{b,c,a}=\bpf{b}\otimes n_Q$ is hyperbolic over the field $F_Q$, and it follows that the $5$-dimensional subform $\varphi_1$ is isotropic over $F_Q$. We now use \cite[Prop 5.2.12]{Faivre:thesis} to check $\varphi_1$ is $F_Q$-minimal. The first condition has already been checked: $\varphi_1$ is a Pfister neighbour of $\pf{b,c,a}=\bpf{b}\otimes n_Q$. Moreover, its Clifford algebra is Brauer equivalent to $[a+b,c)\otimes[a,b)\simeq[a+b,bc)$, and the tensor product with $Q$ is 
\[[a+b,bc)\otimes [a,c)=[a,b)\otimes [b,c),\] which is division as required. This finishes the proof. 
\end{proof}

To get an explicit example, it remains to exhibit a field over which the assumptions of the theorem are fulfilled. This can be done as follows : 
\begin{example} Let $F_0$ be a field of characteristic $2$ and consider the field of iterated Laurent series $F=F_0((c^{-1}))((b^{-1}))((a^{-1}))$. By~\cite[Lem. 3.4]{Chapman:2020BBMS}, the biquaternion algebra $[a,b)\otimes [b,c)$ is division over $F$. Moreover, the symbol $a\dlog b\wedge\dlog c$ is non trivial in $H^3_2(F)$. 
Therefore, $\varphi_1$ and $Q$ defined as in~\cref{dim5.thm} provide an explicit example of an $F_Q$-minimal form. 
\end{example}

\begin{rem}

\begin{enumerate}
\item By contrast, in characteristic $0$, there exists no $5$-dimensional $F_Q$-minimal form over the field of iterated Laurent series $\mc((x))((y))((z))$. Indeed, there is no biquaternion division algebra over this field, see~\cite[XIII, Rmk. 2.5]{Lam}, hence the conditions in~\cite[Prop. 4.1]{HoffmannLewisVanGeel:1992} are not satisfied. 
\item Since $b[1,b]\perp\qf{1}\simeq [1,b]\perp\qf{1}\simeq \mh+\qf{1}$, the form $\varphi_1$ is isometric to 
\[c[1,a+b]\perp b[1,a+b]\perp\qf{1}\simeq c\bigl(\pf{bc,a+b}\perp\qf{c}\bigr).\] 
In characteristic $p$, we may still consider the homogeneous degree $p$ form $\qf{c}\perp n_T$, where $n_T$ now stands for the norm form of $[a+b,bc)_{p,F}$. It is still true in this case that this form is anisotropic and becomes isotropic over the function field of the Severi-Brauer variety of $[a,c)_{p,F}$. This can be checked using~\cite[Thm. 6]{Gille:2000}. 
\end{enumerate}
\end{rem}

The even Clifford algebra of the quadratic form $\varphi_1$, endowed with its canonical involution, may be computed as in~\cite[Ex 2.1]{DolphinQueguiner:2021}. Applying~\cite[Prop. 3.6]{DolphinQueguiner:2017}, we get the following : 
\begin{cor}
Over the field $F=F_0((c^{-1}))((b^{-1}))((a^{-1}))$, with $F_0$ of characteristic $2$, the biquaternion algebra with symplectic involution 
\[(C,\gamma)=\bigl([a+b,c),\bar{\ \,}\bigr)\otimes \bigl([a,b),\bar{\ \,}\bigl)\]
satisfies the following : 
it is anisotropic, becomes hyperbolic over $F_Q$, and does not contain any $\gamma$-stable subalgebra isomorphic to $(Q,\bar{\ }\,)$, where $Q=[a,c)$. 
\end{cor} 
\begin{rem} 
Alternately, using some explicit base change, and the characterization of orthogonal involutions on a split quaternion algebra given in~\cite[Lem 7.1]{Dolphin:2014}, one may check that the even Clifford algebra of $\varphi_1$ is isomorphic to 
\[(C,\gamma)=\bigl([a+b,bc),\bar{\ \,}\bigl)\otimes\bigl(M_2(F),{\mathrm{ad}}_{\bpf{b}}\bigl).\]
The corollary then follows from~\cite[Thm 3.4]{DolphinQueguiner:2017}. 
\end{rem}

The previous example of a $5$-dimensional quadratic form extends to dimension $7$ as follows : 
\begin{prop}
\label{dim7.thm}
Consider the quaternion algebra $Q=[a,d)$ over the field of iterated Laurent series $F=F_0((d^{-1}))((c^{-1}))((b^{-1}))((a^{-1}))$.  The quadratic form 
\[\varphi_2=d[1,a+c]\perp c[1,a+b]\perp bcd[1,a]\perp \qf{1}\] over $F$ is $F_{Q}$-minimal. 
\end{prop} 
\begin{proof}
The first part of the argument follows the same line as in dimension $5$. 
Using the following isometries 
\[d[1,a+c]\perp cd[1,a+c]\simeq d\pf{c,a+c}\simeq d\pf{c,a}\simeq d[1,a]\perp cd[1,a],\] 
and similarly \[c[1,a+b]\perp bc[1,a+b]\simeq c[1,a]\perp bc[1,a],\]
one may check that $\varphi_2$ is a subform of $\qf{1,c,bc}\otimes\pf{d,a}=\qf{1,c,bc}\otimes n_Q$, which in turn is a subform of $\bpf{b,c}\otimes n_Q=\pf{b,c,d,a}$. 
Since $a\dlog b\wedge\dlog c\wedge\dlog d$ is a non-trivial symbol in $H^4_2(F)$, this form is anisotropic, and so is $\varphi_2$. Moreover, over the function field $F_Q$, the forms $n_Q$ and $\qf{1,c,bc}\otimes n_Q$ are hyperbolic, hence $\varphi_2$ is isotropic. 

Denote by $\chi$ the quadratic form $cd[1,a+c]\perp bc[1,a+b]$. Since $[1,a]\perp\qf{1}\simeq \mh\perp\qf{1}$, the computation above actually shows that $\varphi_2\perp \chi$ is Witt equivalent to $\qf{1,c,bc}\otimes n_Q\perp \qf{1}$. 
In the language of~\cite[Def 5.2.3]{Faivre:thesis}, this shows that the invariant $\bar s_{F_Q}(\varphi_2)$ is at most $3$. 
In view of~\cite[Prop 5.2.7]{Faivre:thesis}, to prove $\varphi_2$ is $F_Q$-minimal, it is enough to prove that $\bar s_{F_Q}(\varphi_2)=3$, that is, there is no $4$-dimensional regular quadratic form $\chi$ over $F$ such that 
$\varphi_2\perp \chi$ is Witt equivalent to $\tau\otimes n_Q\perp\qf{1}$ with $\tau\otimes n_Q$ anisotropic, and $\tau$ of dimension $1$ or $2$. 

Let us assume, for the sake of contradiction that such forms exist. Using again $[1,\lambda]\perp\qf{1}\simeq \mh\perp \qf{1}$ for all $\lambda\in F$, we get that
\[\tilde\varphi_2\perp\tilde\chi\perp\qf{1}\sim \tau\otimes n_Q\perp\qf{1},\]
where \[\tilde\varphi_2=d[1,a+c]\perp c[1,a+b]\perp bcd[1,a]+[1,a+b+c]\] and $\tilde\chi$ is the Albert form 
$\chi\perp[1,\delta]$, where $\delta$ is the discriminant of $\chi$. 
By~\cite[Prop 13.6]{EKM}, it follows that 
\[\tilde \varphi_2\perp\tilde \chi\sim\tau\otimes n_Q,\mbox{ with }\tau\mbox{ of dimension $1$ or $2$}.\]

We may now use Clifford algebras to prove $\tau$ has dimension $1$. Indeed, since both forms have trivial discriminant, 
by additivity of the Clifford invariant~\cite[Lem. 14.2]{EKM}, the Witt equivalence above implies that 
$\C(\tilde \varphi_2)\otimes\C(\tilde \chi)$ is split if $\tau$ has dimension $2$ and Brauer equivalent to $Q$ if $\tau$ has dimension $1$.  
The Clifford algebra of $\tilde \varphi_2$ may be computed as in~\cite[Appendix]{DolphinQueguiner:2021}, and we get \[\C(\tilde \varphi_2)\simeq [a+c,d)\otimes[a+b,c)\otimes [a,bcd)\simeq [a,b)\otimes[b,c)\otimes[c,d).\] By \cite[Lem. 3.4]{Chapman:2020BBMS}, it is a division algebra. On the other hand, $\tilde \chi$ is an Albert form, so $\C(\tilde\chi)$ is a biquaternion algebra. Hence the tensor product is not split, and it follows that $\tau$ has dimension $1$. 
Moreover, the Clifford algebra of $\tilde \chi$ is Brauer-equivalent to $\C(\tilde\varphi_2)\otimes Q$ which is  
\[[a+c,d)\otimes[a+b,c)\otimes [a,bc)\sim[a+c,cd)\otimes[a+b,bc),\]
Hence, we get $\tilde \chi\simeq \mu\bigl(cd[1,a+c]+bc[1,a+b]+[1,b+c]\bigr)$ for some $\mu\in F^\times$. 

To finish the proof, we use~\cref{WittIndex.st}. Denote by $v$ the $(d^{-1},c^{-1},b^{-1},a^{-1})$-adic valuation of $F$. Its value group is $\Gamma_F=\mz^4\subset \Gamma=\mq^4$. By~\cite[Prop 3.8 and Ex 3.11 (ii)]{TignolWadsworth:2015}, every $v$-value function on a finite dimensional $F$-vector space is a norm. In particular, the $v$-value functions $\alpha$ and $\beta$ respectively associated to the quadratic forms $\tilde\varphi_2$ and $\tilde \chi$ are norms. Hence, the value set of $\alpha$ is the union of the cosets of values of vectors of a splitting base of $V$. Specifically, one may check that the base of $V$ corresponding to the following presentation of the quadratic form,
\[\tilde \varphi_2\simeq d[1,a+c]\perp c[1,a+b]\perp bcd[1,a]+[1,a+b+c],\]
is a splitting base. It follows that the value set $\Gamma_\alpha$ of the norm $\alpha$ consists of the cosets of the following elements of $\Gamma$:
\[\{\alpha_1;\alpha_1+\alpha';\alpha_2;\alpha_2+\alpha';\alpha_3;\alpha_3+\alpha';\alpha_4;\alpha_4+\alpha'\},\]
\[\mbox{ where }\alpha_1=(\frac12,0,0,0),\  \alpha_2=(0,\frac12,0,0),\ 
\alpha_3=(\frac12,\frac12,\frac12,0),\ \alpha_4=(0,0,0,0),\]
\[\mbox{ and }\alpha'=(0,0,0,\frac12).\]
Any two of those clearly give rise to distinct cosets, so that $|\Gamma_\alpha:\Gamma_F|=8$, which is the dimension of the underlying vector space. Hence the assumptions in~\cref{WittIndex.st} are satisfied and we get that the Witt index of $\tilde\varphi_2\perp\tilde\chi$ is at most 
$|(\Gamma_\alpha\cap\Gamma_\beta):\Gamma_F|$, where $\Gamma_\beta$ denotes the value set of the norm $\beta$. On the other hand, since $\tilde\varphi_2\perp \tilde\chi\sim\tau\otimes n_Q$ with $\tau$ of dimension $1$, this orthogonal sum has Witt index $5$, so that 
\[|(\Gamma_\alpha\cap\Gamma_\beta):\Gamma_F|\geq 5.\]

Besides, we have 
\[\tilde\chi\simeq \mu \bigl(cd[1,a+c]+bc[1,a+b]+[1,b+c]\bigr).\]
Hence, the value set $\Gamma_\beta$ consists of the cosets of the elements 
\[\{\beta_5+\beta'';\beta_5+\beta''+\alpha';\beta_6+\beta'';\beta_6+\beta''+\alpha',\beta'',\beta_7+\beta''\}\]
\[\mbox{ where }\beta_5=(\frac12,\frac12,0,0),\ \beta_6=(0,\frac12,\frac12,0), \ \beta_7=(0,0,\frac12,0)\mbox{ and }\beta''=\frac12\,v(\mu),\]
and $\alpha'$ is as defined above. 

In order to get a contradiction, we use the following : 
\begin{lem}
\label{val.lem}
The element $\frac 12 v(\mu)+\Gamma_F$ is the coset of one of the following : 
\[\{\alpha_4,\ \alpha_4+\alpha', \alpha_2,\alpha_2+\alpha'\}.\]
\end{lem} 
It follows that $\beta''$ in the description of $\Gamma_2$ may only take one of those four values, and in each case, one may check that $\Gamma_\alpha\cap \Gamma_\beta$ consists of at most four distinct cosets, which contradicts the inequality above. 

Hence, to complete the argument, it only remains to prove Lemma~\ref{val.lem}.
Consider the quadratic extension $F'=F[t]/(t^2+t+a)$. There is an isomorphism 
\[F'\simeq F_0((d^{-1}))((c^{-1}))((b^{-1}))((z^{-1})),\mbox{where $z\in F'$ satisfies $z^2+z=a$.}\]
Over $F'$, $n_Q$ is hyperbolic, so $\tilde\varphi_2$ and $\tilde\chi$ are Witt equivalent. Computing the Witt classes of both forms over this field, we get 
\[d[1,c]\perp c[1,b]\perp[1,b+c]\sim\mu\bigl(d[1,c]\perp c[1,b]\perp[1,b+c]\bigr).\]
Hence, $\mu$ is a similarity factor of the Albert form $d[1,c]\perp c[1,b]\perp[1,b+c]$ over the field $F'$. 
By~\cite[(16.6)]{BOI}, it follows $\mu$ belongs to $F'^{\times2}\Nrd_B{B^\times}$, where $B$ is the biquaternion algebra $\bigl([b,c)\otimes [c,d)\bigr)_{F'}$. Since $F$ is henselian, the value $v$ extends uniquely to the field $F'$ and to the division algebra $B$, with value groups respectively denoted by $\Gamma_{F'}$ and $\Gamma_B$. 
Moreover, we have $v_B(x)=\frac 14 v'\bigl(\Nrd_B(x)\bigr)$ for all $x$ in $B$. 
Therefore, $\frac 12 v(\mu)$ belongs to $\Gamma_{F'}+2\Gamma_B$. 
Clearly, we have \[\Gamma_{F'}=\mz^3\times \frac 12\mz\subset\Gamma.\]
To compute the value group of $\Gamma_B$, we proceed as follows. Consider a generating pair $(i,j)$ of $[b,c)_{F'}$ and a generating pair $(k,\ell)$ of $[c,d)_{F'}$ with $i^2+i=b$, $j^2=c$, $ij=j(i+1)$, and similarly $k^2+k=c$, $\ell^2=d$ and $k\ell=\ell(k+1)$. Using those relations, one may easily check that $v_B(i)=(0,0,-\frac12,0)$, $v_B(j)=(0,-\frac 12,0,0)=v_B(k)$ and $v_B(\ell)=(-\frac12,0,0,0)$. Moreover, since $j$ and $k$ commute in $B$, we have $(j+k)^2=k$, hence $v_B(j+k)=(0,-\frac 14,0,0)
.$ Since $B$ has degree $4$, it follows that $B$ is unramified and \[\Gamma_B\simeq \frac 12\mz\times \frac 14\mz\times\frac 12\mz\times\mz\subset\Gamma.\] The assertion on $v(\mu)$ now follows by a direct computation. 
\end{proof}

\bibliographystyle{abbrv}

\end{document}